\documentclass[a4paper,12pt,leqno]{amsart}
\usepackage{amsmath,amsfonts,amsthm,amssymb,dsfont}
\usepackage[alphabetic]{amsrefs}
\usepackage[OT4]{fontenc}
\usepackage{enumerate}
\usepackage[pdftex]{graphicx}
\usepackage{graphics}

\long\def\symbolfootnote[#1]#2{\begingroup
\def\thefootnote{\fnsymbol{footnote}}\footnote[#1]{#2}\endgroup}

\newtheorem{thm}{Theorem}[section]
\newtheorem{cor}[thm]{Corollary}
\newtheorem{lemma}[thm]{Lemma}
\newtheorem{prop}[thm]{Proposition}

\theoremstyle{definition}
\newtheorem{defin}[thm]{Definition}
\newtheorem{rem}[thm]{Remark}
\newtheorem*{rema}{Remark}

\newcommand{\inv}{^{-1}}
\newcommand{\la}{\langle}
\newcommand{\ra}{\rangle}
\newcommand{\ZZ}{\mathbf Z}
\def\Aut{\mathop{\mathrm{Aut}}\nolimits}
\DeclareMathOperator{\Out}{Out}

\title{Twist-rigid Coxeter groups}
\author{}
\date{}

\begin{document}

\maketitle

\begin{center}\bf
Pierre-Emmanuel Caprace$^a$\symbolfootnote[1]{Supported by the Belgian Fund for Scientific Research (FNRS).}
\& Piotr Przytycki$^b$\symbolfootnote[2]{Partially supported by
MNiSW grant N201 012 32/0718 and the Foundation for Polish Science.}
\end{center}

\begin{center}\it
$^a$ Universit\'e catholique de Louvain, D\'epartement de Math\'ematiques, \\
Chemin du Cyclotron 2, 1348 Louvain-la-Neuve, Belgium\\
\emph{e-mail:} \texttt{pe.caprace@uclouvain.be}
\end{center}

\begin{center}\it
$^b$ Institute of Mathematics, Polish Academy of Sciences,\\
\'Sniadeckich 8, 00-956 Warsaw, Poland\\
\emph{e-mail:} \texttt{pprzytyc@mimuw.edu.pl}
\end{center}

\begin{abstract}
\noindent We prove that two angle-compatible Coxeter generating sets
of a given finitely generated Coxeter group are conjugate provided
one of them does not admit any elementary twist. This confirms a
basic case of a general conjecture which describes a potential
solution to the isomorphism problem for Coxeter groups.
\end{abstract}



\section{Introduction}
\label{introduction}

A subset $S$ of a group $W$ is called a \textbf{Coxeter generating
set} if there is a Coxeter matrix $(m_{s,s'})_{s,s'\in S}$ such that
the relations $(ss')^{m_{s,s'}}=1$ provide a presentation of $W$. A
group admitting a Coxeter generating set is called a \textbf{Coxeter
group}. In this article we consider only finitely generated Coxeter
groups; this hypothesis will not be repeated anymore. Any Coxeter
generating set of such a group is automatically finite.

It is a basic and natural problem to determine all possible Coxeter
generating sets for a given Coxeter group $W$. Finding an
algorithmic way to describe these would actually solve the
isomorphism problem for Coxeter groups, which as of today remains
open. Substantial progress in this direction has been accomplished
in recent years (see \cite{MuhlherrSurvey} for a 2006 survey),
providing in particular some important reduction steps which we
shall now briefly outline.

Given a fixed Coxeter generating set $S$ for $W$, an
\textbf{$S$-reflection} (or a \textbf{reflection}, if the dependence
on the generating set $S$ does not need to be emphasised) is an
element conjugate to some element of $S$. Two Coxeter generating
sets $S$ and $R$ for $W$ are called \textbf{reflection-compatible}
if the set of $S$-reflections is contained in the set of
$R$-reflections. Then the sets of $S$- and $R$-reflections coincide
(see Corollary~\ref{cor:ReflCompatible} in Appendix~\ref{App1}). We
shall further say that $S$ and $R$ are \textbf{angle-compatible} if
every \textbf{spherical} pair $\{s, s'\} \subset S$ (\emph{i.e.}\
every pair generating a finite subgroup) is conjugate to some pair
$\{r, r'\} \subset R$ (in particular if we put $s=s'$ this implies
that $S$ and $R$ are reflection-compatible). Saying that $S$ and $R$
are angle-compatible means exactly that $S$ is \emph{sharp-angled}
with respect to $R$, in the language of \cite{MuhlherrSurvey} and
\cite{MarquisMuhlherr}. However, since this implies that conversely
every spherical pair $\{r, r'\} \subset R$ is conjugate to some pair
$\{s, s'\} \subset S$ (see Corollary~\ref{cor:AngleCompatible}), we
prefer a symmetric way of phrasing that.

Reflection- and angle-compatibility are well illustrated by the
simplest class of Coxeter groups, namely finite dihedral groups.
Indeed, the dihedral group of order $20$ is isomorphic to the direct
product of $\ZZ/2$ with the dihedral group of order $10$; the
corresponding Coxeter generating sets are not reflection-compatible.
Moreover, identifying this group with the automorphism group of a
regular decagon, every generating pair of reflections is a Coxeter
generating set, and any two such pairs are reflection-compatible.
However, they are angle-compatible if and only if the associated
pairs of axes intersect at the same angle.

A first motivation to consider the notion of angle-compatibility
comes from the following basic observation.

\begin{rema}
The group $\Aut(W)$ contains a finite index subgroup all of whose
elements map every Coxeter generating set to an angle-compatible
one.
\end{rema}

Indeed, the Coxeter group $W$ has finitely many conjugacy classes of
finite subgroups; in particular, there are finitely many conjugacy
classes of pairs of elements generating a finite subgroup. The
finite index normal subgroup of $\Aut(W)$ preserving the conjugacy
class of each of these pairs satisfies the desired property.

\smallskip
A deeper reason to consider angle-compatible Coxeter generating sets
comes from the fact that, by the main results of
\cite{HowlettMuhlherr} and \cite{MarquisMuhlherr} there are certain
explicit operations which transform any two Coxeter generating sets
into angle-compatible ones (see Appendix~\ref{App2}). It is further
conjectured in \cite{MuhlherrSurvey} that, up to conjugation, any
two angle-compatible Coxeter generating sets may be obtained from
one another by a finite sequence of \emph{elementary twists}, a
notion which was introduced in \cite{BMMN}. If this conjecture is
confirmed, it implies in particular that the isomorphism problem is
decidable in the class of Coxeter groups (see \cite[Corollary
1.1]{MarquisMuhlherr}). The main goal of this paper is to prove a
basic case of this conjecture. In order to provide a precise
statement, we first recall the definition of elementary twists in
detail.

Let $S \subset W$ be a Coxeter generating set. Given a subset $J
\subset S$, we denote by $W_J$ the subgroup of $W$ generated by $J$.
We call $J$ \textbf{spherical} if $W_J$ is finite. If $J$ is
spherical, let $w_J$ denote the longest element of $W_J$. We say
that two elements $s,s'\in S$ are \textbf{adjacent} if $\{s,s'\}$ is
spherical. Given a subset $J \subset S$, we denote by $J^\bot$ the
set of those elements of $S\setminus J$ which commute with $J$. A
subset $J\subset S$ is \textbf{irreducible} if it is not of the form
$K\cup K^\bot$ for some non-empty proper subset $K\subset J$.

Let $J\subset S$ be an irreducible spherical subset and assume that
$S\setminus (J\cup J^\bot)$ is a union of two subsets $A$ and $B$
such that $a$ and $b$ are not adjacent for all $a \in A$ and $b \in
B$. This simply means that $W$ splits as an amalgamated product over
$W_{J \cup J^\bot}$. Note that $A$ and $B$ are in general not
uniquely determined by $J$.

We then consider the map $\tau \colon S \to W$ defined by
$$
\tau(s)= \left\{\begin{array}{ll}
s & \text{if } s \in A,\\
w_J s w_J & \text{if } s \in S \setminus A,
\end{array}\right.
$$
which is called an \textbf{elementary twist}. The relevance of this
notion was first highlighted in \cite[Theorem~4.5]{BMMN}, where it
is shown that any elementary twist $\tau$ transforms $S$ into
another Coxeter generating set for $W$.

\medskip
A Coxeter generating set $S$ is called \textbf{twist-rigid} if it
does not admit any elementary twist. The purpose of this paper is to
study Coxeter groups admitting some twist-rigid Coxeter generating
set; these are called \textbf{twist-rigid} Coxeter groups. Observe
that if $S$ is a Coxeter generating set for $W$ which admits an
elementary twist $\tau$, then $S$ and $\tau(S)$ are two
non-conjugate angle-compatible Coxeter generating sets. Our main
result is the following converse.

\begin{thm}\label{main}
Let $S$ and $R$ be
angle-compatible Coxeter generating sets for a group $W$. If $S$ is twist-rigid, then $S$ and $R$ are conjugate.
\end{thm}

By the Remark above, Theorem~\ref{main} has the
following immediate consequence.

\begin{cor}\label{cor:Out}
If a Coxeter group $W$ is twist-rigid, then $\Out(W)$ is finite.
\end{cor}

Using the main results of \cite{HowlettMuhlherr} and
\cite{MarquisMuhlherr} we also obtain the following, which we prove
in Appendix~\ref{App2}.

\begin{cor}
\label{corollaries to main theorem} Let $W$ be a twist-rigid Coxeter
group.
\begin{enumerate}[(i)]
\item
All Coxeter generating sets for $W$ are twist-rigid.
\item
There is an algorithm which, given a Coxeter matrix (associated with
a Coxeter generating set) for $W$, produces as an output
representatives of all conjugacy classes of Coxeter generating sets
for $W$, in terms of words in the original generators.
\item
In particular, this algorithm produces as an output all possible
Coxeter matrices for $W$. Hence the isomorphism problem is decidable
in the class of twist-rigid Coxeter groups.
\end{enumerate}
\end{cor}

Note that a twist-rigid Coxeter group may admit more than one
conjugacy class of Coxeter generating sets (and even two Coxeter
generating sets with different Coxeter matrices). In fact, the
combination of Theorem~\ref{main} with the main results of
\cite{HowlettMuhlherr} and \cite{MarquisMuhlherr} leads to a precise
description of those Coxeter groups which admit a unique conjugacy
class of Coxeter generating sets. However, we do not formulate this
condition explicitly, since it is technical and not particularly
illuminating. Similarly, we could extract from the same combination
of results a precise description of those twist-rigid Coxeter groups
whose Coxeter generating sets admit only one Coxeter matrix.

Theorem~\ref{main} has been previously proved under various special
assumptions, see \cite{MuhlherrSurvey} for a 2006 state of art (but
note that the announcement of \cite[Theorem 3.7]{MuhlherrSurvey} was
too optimistic at the time). Since then, the only contributions
known to us are \cite{Car} and \cite{RS}.

\subsection*{Outline of the proof strategy}

We now sketch the overall strategy governing the proof of
Theorem~\ref{main}. Our approach is inspired by \cite{MW}. Let $S$
and $R$ be two angle-compatible Coxeter generating sets for $W$. We
call the Davis complex associated with $S$ the \textbf{reference}
Davis complex, and we denote it by $\mathbb A_{\mathrm{ref}}$. The
Davis complex associated with $R$ is called the \textbf{ambient}
Davis complex, and is denoted by $\mathbb A_{\mathrm{amb}}$.

Since $S$ and $R$ are reflection-compatible, the set of
$S$-reflections coincides with the set of $R$-reflections and its
elements are called simply reflections. We denote by $\mathcal{Y}_r$
the wall in $\mathbb A_{\mathrm{ref}}$ fixed by a reflection $r$,
and by $\mathcal{W}_r$ the wall in $\mathbb A_{\mathrm{amb}}$ fixed
by $r$. Since two walls intersect non-trivially if and only if the
associated reflections generate a finite group, it follows that the
assignment $\mathcal Y_{r} \mapsto \mathcal W_{r}$ preserves the
parallelism relation.

The Coxeter generating sets $S$ and $R$ are conjugate if and only if
the set $\{\mathcal W_s\}_{s \in S}$ is \textbf{geometric} in the
sense that it consists of walls containing the panels of some given
chamber of $\mathbb A_{\mathrm{amb}}$. In order to show that
$\{\mathcal W_s\}_{s \in S}$ is geometric, it is enough to construct
a set $\{\Phi_s\}_{s\in S}$ of half-spaces in $\mathbb
A_{\mathrm{amb}}$ satisfying the following. First we require that
the boundary wall of each $\Phi_s$ equals $\mathcal W_{s}$ (shortly,
$\Phi_s$ is a half-space \textbf{for} $s$). Second, we require that
for all $s,s'\in S$ the pair $\{\Phi_s, \Phi_{s'} \}$ is
\textbf{geometric}, \emph{i.e.}\ the set $\Phi_s \cap \Phi_{s'}$ is
a fundamental domain for the $\la s, s'\ra$-action on $\mathbb
A_{\mathrm{amb}}$. The fact that these two conditions imply that
$\{\mathcal W_s\}_{s \in S}$ is geometric is originally due to J.-Y.
H\'ee \cite{Hee} and was established independently by
Howlett--Rowley--Taylor~\cite[Theorem 1.2]{HRT} (see also \cite[Fact
(1.6)]{CM} for yet another proof as well as some additional
references).

In view of the above discussion, proving Theorem~\ref{main} boils down to establishing the following.

\begin{thm}\label{reduced main}
Let $S$ and $R$ be angle-compatible Coxeter generating sets for a group $W$. Assume that $S$ is twist-rigid. Then there exists a
set of half-spaces $\{\Phi_s\}_{s \in S}$ in $\mathbb
A_{\mathrm{amb}}$ such that $\Phi_s$ is a half-space for $s$ and
$\{\Phi_s, \Phi_{s'} \}$ is geometric for all $s, s' \in S$
\end{thm}

For example, if $s$ and $t$ are two elements of $S$ which generate
an infinite dihedral group, we need to define $\Phi_s$ as the unique
half-space for $s$ containing the wall $\mathcal W_{t}$, which we
denote by $\Phi(\mathcal{W}_s, \mathcal{W}_t)$. In particular, if
$t' \in S$ is another element with the same property, we need to
verify that $\mathcal W_{t}$ and $\mathcal W_{t'}$ lie on the same
side of $\mathcal W_{s}$. In order to address this compatibility
issue, we shall view the walls $\mathcal W_{t}$ and $\mathcal
W_{t'}$ as part of a larger family of walls parallel to $\mathcal
W_{s}$. This family is parametrised by a certain set of data which
we call \emph{markings}.

More precisely, a marking $\mu$ with \emph{core} $s \in S$ is a pair
$\mu = ((s,w),m)$, where $w\in W,\ m\in S$, which satisfies a number
of conditions depending on the combinatorics and the geometry of
$\mathbb A_{\mathrm{ref}}$.

We will consider two particular types of markings. One type will be
\emph{complete markings}, which will give rise to walls parallel to
$\mathcal W_{s}$, like in the example above. Namely, we require (in
particular) that the wall $\mathcal{Y}^\mu=w\mathcal{Y}_m$ is
parallel to $\mathcal{Y}_s$, so that
$\mathcal{W}^\mu=w\mathcal{W}_m$ is parallel to $\mathcal{W}_s$.
Hence every complete marking $\mu= ((s,w), m)$ with core $s \in S$
induces a choice of half-space $\Phi^\mu_s =\Phi(\mathcal {W}_s,
\mathcal{W}^\mu)$ in $\mathbb A_{\mathrm{amb}}$.

We will also take advantage of the fact that Theorem~\ref{reduced
main} has been already proved for $2$-spherical Coxeter groups
\cite{CM}. (Recall that $J\subset S$ is called
\textbf{$2$-spherical} if all of its two element subsets are
spherical and a Coxeter group is \textbf{$2$-spherical} if it admits
a Coxeter generating set $S$ which is $2$-spherical.) In particular,
each $2$-spherical but non-spherical subset $J\subset S$ containing
$s$ gives rise to a natural choice of a half-space for $s$. Consider
a marking $\mu=((s,w),m)$ and let $j_1\ldots j_n$ be a word of
minimal length representing $w$. Even if $\mathcal{W}^\mu=j_1\ldots
j_n \mathcal{W}_m$ intersects $\mathcal{W}_s$, once the
\emph{support} $J=\{s, j_1,\ldots, j_n\}\subset S$ is irreducible
$2$-spherical but non-spherical, there is a natural choice of a
half-space $\Phi^\mu_s$ for $s$ (see Corollary~\ref{2-sph
corollary}). A marking $\mu$ of this type is called
\emph{semi-complete}.

We next introduce a relation on the set of all complete and
semi-complete markings with common core, which we call a
\emph{move}. We show that two markings $\mu, \mu'$ related by a move
induce the same choice of a half-space, namely we show $\Phi^\mu_s =
\Phi^{\mu'}_s$. In the special case of complete markings $\mu, \mu'$
the reason for this is that $\mathcal{W}^\mu$ and
$\mathcal{W}^{\mu'}$ intersect.

The markings discussed at the beginning, where
$\mathcal{W}^\mu=\mathcal{W}_t$, for some $t\in S$ (\emph{i.e.}\ markings for which $w=1$) are part of a class of particularly well
behaved \emph{good markings}. They have, in particular, the property
that $w$ is uniquely determined by
the support $J$.

The major part of the proof is then to show that any two good
markings with common core are related to one another by a sequence
of moves. The moves can be essentially read from the Coxeter diagram
of $(W, S)$, and this allows us to appeal to graph-theoretic
arguments. Since $S$ is assumed to be twist-rigid, there is `enough
space' to move around the diagram of $(W, S)$. We use the formalism
of Masur--Minsky hierarchies to assemble all the data.

In the present context, a \emph{hierarchy} is a system of geodesics
in the Coxeter diagram of $(W, S)$ between a pair of good markings.
It is constructed in a seemingly arbitrary way, but later reveals
highly organised structure. It admits a resolution into a sequence
of \emph{slices} which give rise to good markings related by moves.

However, a new type of moves pops up out of this procedure; to
handle it we need to leave the setting of good markings and consider
complete markings in full generality.

\medskip

We stress that the assumption that $S$ is twist-rigid is effectively
used only in two places in the proof. We use it to prove the
existence of a hierarchy with a given main geodesic
(Lemma~\ref{hierachies exist}) and in a similar situation in
Section~\ref{section 2-spherical complete markings}
(Lemma~\ref{avoiding shadow}).

Furthermore, the hypothesis that $S$ and $R$ are angle-compatible
(and not just reflection-compatible) is used only in the proof of
Proposition~\ref{consistence} and in the case where $S=S'\cup
S'^\bot$ for some spherical $S'$.

\subsection*{Organisation of the article}
In Section~\ref{Coxeter} we recall some basic facts on Coxeter
groups and Davis complexes. In Section~\ref{half-space choices} we
define \emph{complete, semi-complete} and \emph{good markings} and
describe how they determine choices of half-spaces. In
Theorem~\ref{well defined} we claim that this choice does not depend
on the marking $\mu$, provided that $\mu$ is good. We next prove
Theorem~\ref{well defined} in Section~\ref{section moves}, by means
of \emph{moves}. However, we leave two crucial graph-theoretic
results, namely Theorem~\ref{connecting markings} and
Proposition~\ref{connecting 2}, for
Sections~\ref{hierarchies}--\ref{section 2-spherical complete
markings}. We prove Theorem~\ref{reduced main} (the main result) in
Section~\ref{choice of half-spaces}.

In Section~\ref{hierarchies} we describe how to connect a pair of
good markings by a \emph{hierarchy}. Then, in Section~\ref{slices},
we show how to resolve a hierarchy into a sequence of \emph{slices},
which gives rise to a sequence of good markings related by moves. In
Section~\ref{section 2-spherical complete markings} we finally deal
with the last type of moves.

We include Appendix~\ref{App1}, where we explain why the relations
of reflection- and angle-compatibility are symmetric. In Appendix
\ref{App2} we provide a proof of Corollary~\ref{corollaries to main
theorem}.

\subsection*{Acknowledgements}
The writing of this paper was initiated during the stay of both authors at the
Institut des Hautes \'Etudes Scientifiques, and continued
at the Hausdorff Research Institute for Mathematics. We are grateful for the hospitality of both institutes.

\section{Preliminaries on Coxeter groups}
\label{Coxeter}

In this section we collect some basic facts on Coxeter groups. Let
$W$ be a group with a Coxeter generating set $S$ and let $\mathbb A$
be the associated Davis complex.

The gallery distance between chambers $c,c'$ of the Davis complex
$\mathbb A$ is denoted by $d(c,c')$.
By the \textbf{distance} of a chamber $c$ to a wall $\mathcal Y$, we
mean the minimal gallery distance from $c$ to a chamber containing a
panel contained in $\mathcal Y$; it is denoted by $d(c, \mathcal
Y)$.

If $w\in W$ and $c$ is chamber of $\mathbb A$, we denote by $w.c$
the image of $c$ under the action of $w$. For $w\in W$ let $\ell(w)$
denote the word-length of $w$; in other words, $\ell(w)=d(w.c_0,
c_0)$, where $c_0$ is the identity chamber of $\mathbb A$.

\begin{lemma}[{\cite[Chapitre IV, Exercice 22]{Bou}}]
\label{bourbaki} Let $w\in W$. The set of $s\in S$ satisfying
$\ell(ws)<\ell(w)$ is spherical.
\end{lemma}

For $s\in S$, we denote by $\alpha_s$ the positive root in the Tits
representation corresponding to $s$ (see \emph{e.g.}\ \cite[Section
1]{BH}).

\begin{lemma}[{\cite[Lemma 1.7]{BH}}]
\label{brink-howlett} Let $\mathcal{Y}$ be a wall with associated
positive root $\alpha$ and let $s\in S$. Then we have
\begin{equation*}
\text{d}(s.c_0,\mathcal{Y})=
\begin{cases} d(c_0,\mathcal{Y}) +1& \text{if } \langle \alpha, \alpha_s \rangle >0,
\\
\text{d}(c_0,\mathcal{Y}) &\text{if } \langle \alpha, \alpha_s
\rangle =0,
\\
\text{d}(c_0,\mathcal{Y})-1 &\text{if } \langle \alpha, \alpha_s
\rangle <0.
\end{cases}
\end{equation*}
\end{lemma}

The following is a consequence of \cite[Proposition 5.5]{Deodhar}
(see also \cite[Proposition 3.1.9]{Kra}).

\begin{prop}
\label{deod main} If $I\subset S$ is irreducible and non-spherical,
then its centraliser in $W$ coincides with $W_{I^\bot}$.
\end{prop}

We also need the following, known as the Parallel Wall Theorem.

\begin{thm}[{\cite[Theorem 2.8]{BH}}]
\label{parallel wall theorem main} For any $(W,S)$ there exists a
constant $n$ such that the following holds. For any wall
$\mathcal{Y}$ and a chamber $c$ at gallery distance at least $n$
from $\mathcal{Y}$, there is another wall separating $c$ from
$\mathcal{Y}$.
\end{thm}

We finish this section with the following known fact; as explained
in the introduction, it is a basic ingredient of the proof of
Theorem~\ref{main}.

\begin{thm}[{\cite[Main Result (1.1)]{CM}}]
\label{2-spherical}
Let $S$ and $R$ be
reflection-compatible Coxeter generating sets for a group $W$. If $S$ is irreducible $2$-spherical and non-spherical, then $S$ and $R$ are conjugate.
\end{thm}

By Theorem~\ref{small Coxeter groups} this yields the following.

\begin{cor}
\label{2-sph corollary} Let $R$ be a Coxeter generating set for $W$.
Let $S\subset W$ be a Coxeter generating set for a subgroup
$W_S\subset W$, consisting of $R$-reflections. If $S$ is irreducible
$2$-spherical and non-spherical, then there is a unique fundamental
domain for the $W_S$-action on the Davis complex of $W$ associated
with $R$, which is adjacent to all $s$-invariant walls, over $s\in
S$. In particular, if $\Phi_s$, for $s\in S$, denotes the half-space
for $s$ containing this fundamental domain, then for every
$\{s,s'\}\subset S$ the pair $\{\Phi_s,\Phi_{s'}\}$ is geometric.
\end{cor}

\section{Choices of half-spaces}
\label{half-space choices}

Recall that we simultaneously work in the reference Davis complex
$\mathbb A_{\mathrm{ref}}$ associated with the generating set $S$,
where the wall fixed by a reflection $r\in W$ is denoted by
$\mathcal{Y}_r$, and in the ambient Davis complex $\mathbb
A_{\mathrm{amb}}$, where the wall fixed by $r$ is denoted by
$\mathcal{W}_r$. The word-length $\ell(\cdot)$ will be always
measured with respect to the set $S$.

The aim of this section is to introduce the notions of
\emph{complete}, \emph{semi-complete} and \emph{good markings}. Each
such marking $\mu$ with \emph{core} $s\in S$ will determine a
half-space $\Phi_s^\mu$ for $s$ in the ambient Davis complex
$\mathbb A_{\mathrm{amb}}$.

First, we give a rough definition of a \emph{complete marking}. We
start at $c_0$, the identity chamber of $\mathbb A_{\mathrm{ref}}$.
We consider a gallery issuing from $c_0$ which moves away from
$\mathcal Y_s$. We stop at the first wall we cross which does not
intersect $\mathcal Y_s$, and we denote it by $\mathcal Y^\mu$. If
the type of the last panel is $m\in S$, then $\mathcal
Y^\mu=w\mathcal Y_m$, where $w.c_0$ is the last chamber of our
gallery. We call $m$ the \emph{marker} and $(s,w)$ the \emph{base}.

We now give the precise definitions.

\begin{defin}
\label{domain} A \textbf{base} is a pair
$(s,w)$ with $s\in S$ and $w\in W$ satisfying
\begin{enumerate}[(i)]
\item
$d(w.c_0,\mathcal{Y}_{s})=\ell(w)$, and
\item
every wall which separates $w.c_0$ from $c_0$ intersects $\mathcal{Y}_s$.
\end{enumerate}

Note that condition (i) is equivalent to $d(w.c_0,
sw.c_0)=2\ell(w)+1$. We call $s$ the \textbf{core} of the base. The
\textbf{support} of the base is the smallest subset $J\subset S$
such that $W_J$ contains both $s$ and $w$. A base is
\textbf{irreducible, spherical, $2$-spherical}, \emph{etc}, if its
support is irreducible, spherical, $2$-spherical, \emph{etc}.
\end{defin}
\begin{rem}
\label{extending base}
Let $(s,w)$ be a base and let $j_1\ldots j_n$ be a word (over $S$) of minimal length representing $w$. Then the support $J$ of $(s,w)$ equals
$\{s,j_1,\ldots, j_n\}$.
\begin{enumerate}[(i)]
\item By Lemma~\ref{brink-howlett}, condition (i) in Definition
\ref{domain} is equivalent to $\langle j_{i-1}\ldots
j_1\alpha_{s}, \alpha_{j_i}\rangle<0,$ for all $1\leq i\leq n$.
(Condition (ii) is equivalent to this expression being greater than
$-1$.)
\item
In particular, if we write $\alpha^i=j_i\ldots j_1\alpha_{s}=\sum_{j\in J}\lambda^i_j\alpha_j$, then, since
$\alpha^i=\alpha^{i-1}-2\langle\alpha^{i-1}, \alpha_{j_i}\rangle \alpha_{j_i}$, the coefficients $\lambda^i_j$ are nondecreasing in $i$.
\item
By (i), if $j\in S\setminus (J\cup J^\bot)$, then $(s,wj)$
satisfies condition (i) in Definition~\ref{domain}.
\end{enumerate}
\end{rem}

\begin{rem}
\label{parallel wall theorem} By Theorem~\ref{parallel wall theorem
main}, for fixed $(W,S)$ the value of $\ell(w)$ in
Definition~\ref{domain} is bounded.
\end{rem}

Below we show that the support of a base must be $2$-spherical of a very specific type.

\begin{defin}
A $2$-spherical subset $J\subset S$ is \textbf{tree-$2$-spherical}
if its Dynkin diagram is a union of trees.
\end{defin}

\begin{lemma}
\label{domain tree-2-spherical} Any base $(s,w)$ is irreducible
tree-$2$-spherical (\emph{i.e.}\ its support $J$ is irreducible
tree-$2$-spherical).
\end{lemma}
\begin{proof}
Let $j_1\ldots j_n$ be a word of minimal length representing $w$.
If $J=\{s,j_1,\ldots, j_n\}$ is reducible, and $i$ is the least index for which $\{s,
\ldots j_i\}$ is reducible, then the distances from
$j_1\ldots j_{i-1}.c_0$ and $j_1\ldots j_i.c_0$ to
$\mathcal{Y}_{s}$ are equal, violating condition (i) in Definition
\ref{domain}.

If $J$ is not tree-$2$-spherical, then let $i$ be the least index such that
$J'=\{s,j_1,\ldots j_i\}$ is not tree-$2$-spherical. We claim that
$j_1\ldots j_{i-1}\mathcal{Y}_{j_i}$ does not intersect
$\mathcal{Y}_{s}$, contradicting condition (ii) in Definition
\ref{domain}.

Indeed, by Remark~\ref{extending base}(ii) one can show inductively
that all $\lambda^{i-1}_j$, for $j\in J'\setminus\{j_i\}$, equal at
least $1$. Since $J'$ is not tree-$2$-spherical, there are at least
two elements $j\in J'\setminus\{j_i\}$ such that $\langle
\lambda^{i-1}_j\alpha_j,\alpha_{j_i}\rangle\leq -\frac{1}{2}$, or at
least one with $\langle
\lambda^{i-1}_j\alpha_j,\alpha_{j_i}\rangle\leq -1$. Hence $\langle
\alpha^{i-1}, \alpha_{j_i}\rangle\leq -1$ and consequently $\mathcal
Y_{j_i}$ does not intersect $j_{i-1}\ldots j_1\mathcal Y_s$, as
required.
\end{proof}

Throughout most of the article we will be only discussing the
following special kind of a base.

\begin{defin}
\label{defin of simple}
Assume that $(s,w)$ is a base satisfying $w=j_1\ldots j_n$ where all $j_i$ are pairwise different and different from $s$. We call such a base \textbf{simple}.

If $J\subset S$ is irreducible spherical and $s\in J$, then there
exists a simple base with support $J$ and core $s$. Namely, it
suffices to order the elements of $J\setminus \{s\}$ into a sequence
$(j_i)$ so that for every $1\leq i\leq n$ the set $\{s, j_1,\ldots,
j_i\}$ is irreducible. Every $(s,j_1\ldots j_i)$ is a base by
inductive application of Remark~\ref{extending base}(iii).
\end{defin}

\begin{lemma}
\label{simple} Two simple bases with common core and common support
are equal.
\end{lemma}
\begin{proof}
Let $(s,w)$ and $(s,w')$ be two simple bases with common core and
support. Let $w=j_1\ldots j_n$ and $w'=j_{\pi(1)}\ldots j_{\pi(n)}$,
where $\pi$ is a permutation of the set $\{1,\ldots, n\}$. To
re-order the $j_i$'s and prove $w'=w$ it suffices to show that if
for some $1<i\leq n$ the element $j_{\pi(i-1)}$ does not commute
with $j_{\pi(i)}$, then we have $\pi(i-1)<\pi(i)$.

First observe that, by condition (i) in Definition~\ref{domain}, the
sets $T_i=\{s,j_1,\ldots, j_i\}$ and $T'_i=\{s,j_{\pi(1)},\ldots,
j_{\pi(i)}\}$ are irreducible for every $1\leq i\leq n$ (as in the
proof of Lemma~\ref{domain tree-2-spherical}). By Lemma~\ref{domain
tree-2-spherical} the Dynkin diagram of $J$ is a tree, hence the
Dynkin diagrams of all the $T_i$ and $T'_i$ are subtrees.

If $j_{\pi(i-1)}$ does not commute with $j_{\pi(i)}$, then
$j_{\pi(i-1)}$ separates $j_{\pi(i)}$ from $s$ in the tree
corresponding to $T'_i$. In particular $j_{\pi(i-1)}$ separates
$j_{\pi(i)}$ from $s$ in the tree corresponding to the entire $J$.
Hence $j_{\pi(i-1)}$ belongs to $T_{\pi(i)}$ and consequently we
have $\pi(i-1)<\pi(i)$, as desired.
\end{proof}

Finally, we define a \emph{complete marking}.

\begin{defin}
\label{marking} A \textbf{marking} is a pair $((s,w),m)$, where
$(s,w)$ is a base (the \textbf{base} of the marking) and
$m\in S\setminus J^\bot$ ($m$ is called the \textbf{marker}), where $J$ is the support of $(s,w)$.

We say that the marking is \textbf{complete}, if $w\mathcal Y_m$ does not intersect $\mathcal Y_s$, \emph{i.e.}\
$(s,wm)$ does not satisfy condition (ii) in
Definition~\ref{domain}.

The \textbf{core} and the \textbf{support} of the marking $\mu$ are
the core and the support of its base. The marking $\mu$ is \textbf{simple} if
its base is simple and $m\notin J$.
\end{defin}

\begin{rem}
\label{extending marking}~\\
\vspace{-.3cm}
\begin{enumerate}[(i)]
\item
By Remark~\ref{extending base}(ii), if $((s,w),m)$ is a complete
marking and $j\in S$ is such that $\ell(wj)>\ell(w)$ and $(s,wj)$ is
a base, then $((s,wj),m)$ is a complete marking.
\item
In particular, in view of Remark~\ref{extending base}(iii), we have the
following. If $((s,w),m)$ is a complete marking with support $J$, and $j\in
S\setminus (J\cup J^\bot)$ is such that $((s,w),j)$ is not a
complete marking, then $((s,wj),m)$ is a complete
marking.
\end{enumerate}
\end{rem}

We describe how complete markings with core $s$ determine half-spaces for $s$ in
the ambient Davis complex $\mathbb A_{\mathrm{amb}}$.

\begin{defin}
\label{half-space} Let $\mu=((s,w),m)$ be a complete marking with
core $s$. Denote $\mathcal{W}^\mu=w\mathcal{W}_m$. We define
$\Phi_s^\mu=\Phi(\mathcal{W}_s, \mathcal{W}^\mu)$ (which, as in the
introduction, denotes the half-space for $s$ containing
$\mathcal{W}^\mu$ in $\mathbb A_{\mathrm{amb}}$).
\end{defin}

There is another way to determine half-spaces for $s$.

\begin{defin}
\label{half-space semi-complete} A marking $((s,w),m)$ with support
$J$ is \textbf{semi-complete} \sloppy if $J\cup \{m\}$ is
irreducible $2$-spherical but non-spherical. We then define
$\Phi_s^{\mu}$ to be the half-space for $s$ in $\mathbb
A_{\mathrm{amb}}$ given by Corollary~\ref{2-sph corollary}. If a
marking is at the same time complete and semi-complete, then by
Corollary~\ref{2-sph corollary} this coincides with
Definition~\ref{half-space}.
\end{defin}

Note that a complete marking might not be semi-complete; we have
decided to use this term to underline the fact that we are treating
complete and semi-complete markings similarly.

\begin{rem}
\label{extra remark} If $(s,w)$ is a base with support $J$ and $m\in
S\setminus J^\bot$ is such that $J\cup \{m\}$ is non-spherical, then
$((s,w),m)$ is a semi-complete or complete marking. This follows
from the fact that if $J\cup \{m\}$ is not $2$-spherical, then by
Remark~\ref{extending base}(ii) the wall $\mathcal Y_m$ does not
intersect $w^{-1}\mathcal Y_s$.
\end{rem}

\begin{defin}
Finally, a \textbf{good} marking is a complete or semi-complete
simple marking with spherical base.
\end{defin}

We point out that, under mild hypothesis, good markings exist.

\begin{lemma}
\label{good markings exist} Let $S$ be an irreducible non-spherical
Coxeter generating set for $W$. Then for every $s\in S$ there exists
a good marking with core $s$.
\end{lemma}
\begin{proof}
Let $J\subset S$ be a maximal irreducible spherical subset
containing $s$. Any simple marking with support $J$ and with marker
in $S\setminus (J\cup J^\bot)$ is good, by Remark~\ref{extra
remark}.
\end{proof}

The main element of the proof of Theorem~\ref{reduced main} is the
following.

\begin{thm}
\label{well defined} Let $S$ and $R$ be reflection-compatible
Coxeter generating sets for $W$. Assume that $S$ is twist-rigid. Let
$s\in S$. We consider all (if there are any) good markings $\mu$
with core $s$. Then the half-space $\Phi_s^{\mu}$ (in $\mathbb
A_{\mathrm{amb}}$) does not depend on $\mu$.
\end{thm}

We explain the proof and the consequences of Theorem~\ref{well
defined} in the next sections.

Observe that in Theorem~\ref{well defined} we only assume that $S$
and $R$ are reflection-compatible and we do not require them to be
angle-compatible.

Although in the statement of Theorem~\ref{well defined} we consider
only a restricted family of markings, namely the good markings, the
other more general markings will come up in the proof.

\section{Moves}
\label{section moves}

In this Section we describe the ingredients of the proof of
Theorem~\ref{well defined}.

\begin{defin}
Let $((s,w),m), ((s,w'),m')$ be complete or semi-complete \sloppy
markings with common core and supports $J,J'$. We say that they are
related by \textbf{move}
\begin{enumerate}
\item[\textbf{M1}] if $w=w'$, the markers $m$ and $m'$ are adjacent, and both
markings are complete,
\item[\textbf{M2}] if there is $j\in S$ such that $w=w'j$ and
moreover $m$ equals $m'$ and is adjacent to $j$.
\item[\textbf{M3}] if $J\cup\{m\}\cup J'\cup\{m'\}$ is
$2$-spherical,
\item[\textbf{M4}] if $((s,w),m)$ is complete, for some maximal
spherical subset $K$ of $J$ we have $K\subset m^\bot$, and $J=J'\cup
\{m'\}$.
\end{enumerate}
\end{defin}

The half-spaces $\Phi_s^\mu$ below are chosen in $\mathbb
A_{\mathrm{amb}}$, as in Definitions~\ref{half-space} and
\ref{half-space semi-complete}.

\begin{lemma}
\label{moves do not change half-space} If markings $\mu$ and $\mu'$
are related by one of moves M1--M4, then $\Phi_s^\mu=\Phi_s^{\mu'}$,
where $s$ is the common core of $\mu$ and $\mu'$.
\end{lemma}

\begin{proof} The argument depends on the type of the move.
\begin{enumerate}
\item[\textbf{M1}]
Since $\mathcal{W}_m$ intersects
$\mathcal{W}_{m'}$, it follows that $\mathcal{W}^\mu=w\mathcal{W}_m$
intersects $\mathcal{W}^{\mu'}=w\mathcal{W}_{m'}$. They are both
disjoint from $\mathcal{W}_{s}$, hence they lie in the same
half-space for $s$.

\item[\textbf{M2}]
If any of the markings is not complete, then they are also related
by move M3, see below.

Assume now that both markings are complete. Since $\mathcal{W}_j$
intersects $\mathcal{W}_m$, it follows that $j\mathcal{W}_m$
intersects $\mathcal{W}_m$. Hence
$\mathcal{W}^{\mu'}=w'\mathcal{W}_m=wj\mathcal{W}_m$ intersects
$\mathcal{W}^\mu=w\mathcal{W}_m$ and they lie in the same half-space
for $s$.

\item[\textbf{M3}] This case
follows immediately from Corollary~\ref{2-sph corollary}.

\item[\textbf{M4}]
Since $J$ is non-spherical and $K\subset J$ is maximal spherical,
the fixed point set of $wW_Kw^{-1}$ in $\mathbb A_{\mathrm{amb}}$ is
disjoint from $\mathcal{W}_s$ and is contained in the half-space
$\Phi^{\mu'}_s$ (Corollary~\ref{2-sph corollary}). Moreover,
$\mathcal{W}^\mu=w\mathcal{W}_{m}$ intersects this fixed point set,
hence we have $\Phi^\mu_s=\Phi^{\mu'}_s$.
\end{enumerate}
\end{proof}

\medskip
In view of Lemma~\ref{moves do not change half-space}, in order to
prove Theorem~\ref{well defined}, it is enough to prove the
following.

\begin{thm}
\label{connected by moves} Let $S$ be a twist-rigid Coxeter
generating set for $W$. Let $\mu$ and $\mu'$ be good markings with common
core $s\in S$. Then there is a sequence, from $\mu$ to $\mu'$, of
complete or semi-complete markings such that each two consecutive
ones are related by one of moves M1--M4.
\end{thm}

Observe that this is a statement concerning only the reference Davis
complex $\mathbb A_{\mathrm{ref}}$. The proof of
Theorem~\ref{connected by moves} consists of two pieces. The first
one is the following, which we prove in
Sections~\ref{hierarchies}--\ref{slices}.

\begin{defin}
Let $((s,w),m)$ and $((s,w'),m')$ be good markings with common core.
We say that they are related by \textbf{move N1} if $w=w'$ and $m$
and $m'$ are adjacent.
\end{defin}

\begin{thm}
\label{connecting markings} Let $(W,S),\mu, \mu'$ be as in Theorem
\ref{connected by moves}. Then there is a sequence, from $\mu$ to
$\mu'$, of good markings such that each two consecutive ones are
related by move N1, M2 or M3.
\end{thm}

The second ingredient is the following, which helps us to resolve
move N1. We prove it in Section~\ref{section 2-spherical complete
markings}.

\begin{prop}
\label{connecting 2} Let $S$ be a twist-rigid Coxeter generating set
for $W$. Let $\mu$ be a complete marking with
non-spherical base. Then there is a sequence of complete markings
from $\mu$ to a semi-complete (possibly not complete) marking
$\mu'$ with support $J'$ containing $J$ satisfying the following. Each two
consecutive markings in the sequence are related by move M1, M2 or M4, where move M4 may
appear only as the last one.
\end{prop}

Notice that since $\mu$ and $\mu'$ are related by moves, they have
in particular the same core.

\medskip

We demonstrate how those two ingredients fit together to form the
following.

\begin{proof}[Proof of Theorem~\ref{connected by moves}]
By Theorem~\ref{connecting markings}, it suffices to prove
Theorem~\ref{connected by moves} under the assumption that good
markings $\mu$ and $\mu'$ are related by move N1. If both $\mu,
\mu'$ are semi-complete, then they are related by move M3. On the
other hand, if they are both complete, then they are related by move
M1. Hence without loss of generality we can restrict to the case where $\mu=((s,w),m)$
is complete and $\mu'=((s,w),m')$ is not complete.

By Remark~\ref{extending marking}(ii), the pair $\nu=((s,wm'),m)$ is
another complete marking and it is related to $\mu$ by move M2. The
base of $\nu$ is non-spherical.

Now we apply Proposition~\ref{connecting 2} to $\nu$. We obtain a
semi-complete marking $\nu'$, related to $\nu$ by a sequence of
moves M1, M2 and M4, with support containing $J\cup \{m'\}$, where
$J$ is the support of $\mu$. Hence $\nu'$ is related to $\mu'$ by
move M3.
\end{proof}

\begin{rem}
\label{removing general}
The hypothesis that $S$ is twist-rigid in both Theorem~\ref{connecting markings} and Proposition~\ref{connecting 2} can be weakened, see Remarks~\ref{removing assumptions hierarchies} and~\ref{removing assumptions hierarchies 2}.
\end{rem}
\section{Proof of the main theorem}
\label{choice of half-spaces}

In this section we deduce Theorem~\ref{reduced main} from
Theorem~\ref{well defined}. Before we do that we point out the
following, which does not require a proof (see Figure 1).

\begin{lemma}
\label{fusion} Let $\{s,s'\}\subset W$ be conjugate to some spherical non-commuting pair $\{r,r'\}\subset R$.
Suppose a wall $\mathcal{W}$ (in $\mathbb A_{\mathrm{amb}}$) intersects at least one of
$\mathcal{W}_s,\mathcal{W}_{s'}$ and none of $s\mathcal{W}_{s'},
s'\mathcal{W}_{s}$. Then the pair
$$\{s\Phi(s\mathcal{W}_{s'},\mathcal{W}),
s'\Phi(s'\mathcal{W}_{s},\mathcal{W})\}$$ is geometric. (If
$m_{s,s'}=3$, then we do not need the hypothesis that $\mathcal{W}$
intersects at least one of $\mathcal{W}_s,\mathcal{W}_{s'}$.)
\end{lemma}

\begin{figure}[h]
\includegraphics[width=7cm]{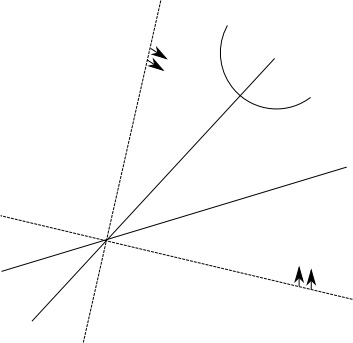}
\put(-27,143){$\mathcal W$} \put(-214,37){$\mathcal W_s$}
\put(-192,0){$\mathcal W_{s'}$} \put(-55,50){$\Phi(s\mathcal W_{s'},
\mathcal W)$} \put(-179,158){$\Phi(s'\mathcal W_{s}, \mathcal W)$}
\caption{Configuration of walls from Lemma~\ref{fusion}} \label{fig}
\end{figure}
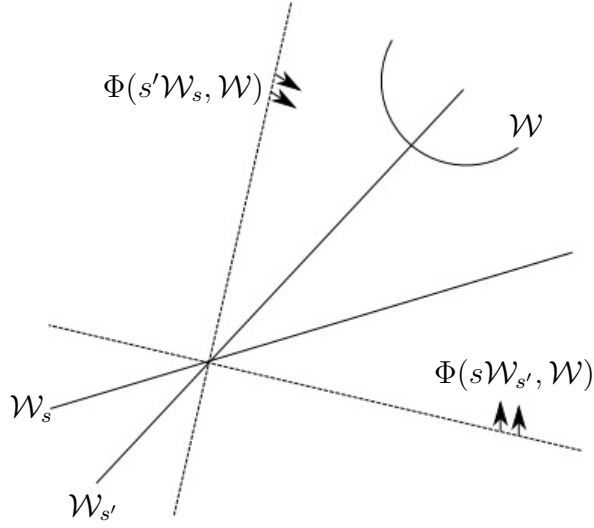

\begin{proof}[Proof of Theorem~\ref{reduced main}] Without
loss of generality we may assume that $S$ is irreducible. If $S$ is
spherical, then the theorem follows from \cite[Proposition 11.7]{CM}.

Otherwise, since $S$ is non-spherical, by Lemma~\ref{good markings
exist} each $s\in S$ is a core of a good marking $\mu$. Hence we can
put $\Phi_s=\Phi_s^{\mu}$ and by Proposition~\ref{well defined} this
does not depend on the choice of $\mu$. It remains to prove the
following. We stress that we do not need to assume anymore that $S$
is twist-rigid.

\begin{prop}
\label{consistence} Let $S$ and $R$ be angle-compatible Coxeter
generating sets for $W$. Assume that $S$ is irreducible and
non-spherical, and let $s,s'\in S$. Suppose that there are
half-spaces $\Phi_s, \Phi_{s'}$ for $s,s'$ in $\mathbb
A_{\mathrm{amb}}$, such that for all good markings $\mu,\mu'$ with
respective cores $s,s'$ we have $\Phi_s=\Phi_s^{\mu}$ and
$\Phi_{s'}=\Phi_{s'}^{\mu'}$. Then the pair $\{\Phi_s, \Phi_{s'}\}$
is geometric.
\end{prop}

\begin{proof} If $s$ and $s'$ are not adjacent, then we can consider
complete markings $\mu=((s,1),s'), \ \mu'=((s',1),s)$ and we obtain
$\Phi_s^{\mu}=\Phi (\mathcal{W}_s,\mathcal{W}_{s'})$,
$\Phi_{s'}^{\mu'}=\Phi (\mathcal{W}_{s'},\mathcal{W}_s)$, as
desired. Hence we may assume that $s,s'$ are adjacent. We may also
assume that they do not commute.

Denote the union of $\{s\}$ with the set of all elements from $S$
adjacent to $s$ by $B(s)$. If there is $t\in S$ outside $B(s)\cup
B(s')$, then we proceed as follows. Let $\Sigma$ be the union of the
two acute-angled sectors between $\mathcal{W}_s$ and
$\mathcal{W}_{s'}$. Since the choices of half-spaces coming from the
markings $((s,s'),t)$ and $((s,1),t)$ coincide, it follows that
$\mathcal{W}_t$ is contained in $\Sigma\cup s'\Sigma$. Analogously,
since the choices of half-spaces coming from the markings
$((s',s),t)$ and $((s',1),t)$ coincide, $\mathcal{W}_t$ is contained
in $\Sigma\cup s\Sigma$. Since we have $(\Sigma\cup s'\Sigma)\cap
(\Sigma\cup s\Sigma)=\Sigma$, we obtain $\mathcal{W}_t\subset
\Sigma$, and consequently $\{\Phi(\mathcal W_s, \mathcal W_t),
\Phi(\mathcal W_{s'},\mathcal W_t)\}$ is geometric, as desired. We
assume henceforth $S=B(s)\cup B(s')$.

Moreover, if there is an irreducible $2$-spherical but non-spherical
subset $J$ of $S$ containing $s$ and $s'$, then we take some maximal
irreducible spherical $K\subset J$ containing $s,s'$ and any $m\in
J\setminus (K\cup K^\bot)$. We consider semi-complete simple
markings $\mu, \mu'$ with support $K$, marker $m$, and cores $s,s'$.
The pair of half-spaces $\{\Phi^\mu_s, \Phi^{\mu'}_{s'}\}$ is
geometric by Corollary~\ref{2-sph corollary}. Hence we can assume
from now on that any irreducible $2$-spherical subset $J$ of $S$
containing $s$ and $s'$ is spherical.

\medskip\par\noindent\textbf{Claim.}\ignorespaces \
There exist complete simple markings $\mu=((s,s'j_2\ldots j_n),m)$,
$\mu'=((s',sj_2\ldots j_n), m)$ with common spherical support
$J=\{s,s',j_2,\ldots, j_n\}$ such that the common marker $m$
satisfies the following.

At least one of $s,s'$ commutes with $\{j_2,\ldots, j_n\}$ and its
invariant wall (\emph{i.e.}\ $\mathcal{W}_s$ or $\mathcal{W}_{s'}$)
intersects $\mathcal{W}_m$.
\medskip

Before we justify the claim, let us show how it implies the
proposition. We verify the hypothesis of Lemma~\ref{fusion} for
$\mathcal{W}=j_2\ldots j_n\mathcal{W}_m$. If, say, $s$ commutes with
$\{j_2,\ldots, j_n\}$ and $\mathcal{W}_s$ intersects
$\mathcal{W}_m$, then $\mathcal{W}_s$ also intersects $\mathcal{W}$.

On the other hand, since $\mu$ and $\mu'$ are complete,
$\mathcal{W}$ does not intersect $s\mathcal{W}_{s'}$ and
$s'\mathcal{W}_s$. Hence, by Lemma~\ref{fusion}, the pair formed by
$\Phi_{s}^\mu=\Phi(\mathcal{W}_s,
s'\mathcal{W})=s'\Phi(s'\mathcal{W}_{s}, \mathcal{W})$ and
$\Phi_{s'}^{\mu'}=\Phi(\mathcal{W}_{s'},
s\mathcal{W})=s\Phi(s\mathcal{W}_{s'}, \mathcal{W})$ is geometric.

\medskip

We now justify the claim. If $B(s)\neq B(s')$, then this is obvious,
we take $J=\{s,s'\}$ and $m$ outside $B(s)\cap B(s')$. Otherwise, we
pick a maximal irreducible spherical subset $K\subset S$ containing
$s,s'$ and an element $m\in S\setminus (K\cup K^\bot)$. By our
discussion $m$ is adjacent to both $s$ and $s'$, but not adjacent to
some $t\in K$. Let $J\subset K$ be the union of $\{s,s'\}$ with the
set of all vertices in the $t$ component of the Dynkin diagram of
$K\setminus \{s,s'\}$. Either $s$ or $s'$ is a leaf in the Dynkin
diagram of $J$. This implies the claim.
\end{proof}

\medskip
This ends the proof of Theorem~\ref{reduced main}. However, we still
need to prove Theorem~\ref{connecting markings} and Proposition
\ref{connecting 2}, which we do in the remaining sections.
\end{proof}

\section{Hierarchies}
\label{hierarchies}

Our goal for this and the next section is to prove Theorem
\ref{connecting markings}. First we need to assemble the
connectivity data of the Coxeter diagram of $(W,S)$, and we do it
via the \emph{hierarchy} formalism. This formalism was invented in a
different context by Masur--Minsky \cite[Section 4]{MM}. Where
convenient, we preserve the original names, notation and structure
of the exposition.

The core of all our markings, throughout this and the next section
is a fixed $s\in S$, and all markings are simple with spherical
bases. Hence any marking is uniquely determined by its irreducible
spherical support $J\ni s$ and its marker $m\in S\setminus (J\cup
J^\bot)$ (see Definition~\ref{defin of simple} and
Lemma~\ref{simple}). Hence we may allow ourselves to write $(J,m)$
instead of $((s,w), m)$. In this and the next section, the only
place where we will use the hypothesis that $S$ is twist-rigid, will
be in the proof of Lemma~\ref{hierachies exist}.

Below, a \textbf{path} in $T\subset S$ is a sequence of elements
from $T$ such that each two consecutive ones are adjacent. A path is
\textbf{geodesic} in $T$ if its length is minimal among paths in $T$
with the same endpoints.

\begin{defin}[{compare \cite[Definition 4.2]{MM}}]
Let $J\subset S$ be irreducible spherical. A \textbf{geodesic} $k$
with \textbf{domain} $J$ is a triple $$k=((k_0, \ldots k_n), I_k,
T_k),$$ where $(k_0,\ldots,k_n)$ is a geodesic path in $S\setminus
(J\cup J^\bot)$, with $n\geq 0$, and
$I_k=(J_{I_k},m_{I_k}),T_k=(J_{T_k},m_{T_k})$ are markings
satisfying the following.

We require that either $I_k=(J,k_0)$ or $I_k$ is good and $J\cup
\{k_0\}\subset J_{I_k}$. Similarly, we require that either
$T_k=(J,k_n)$ or $T_k$ is good and $J\cup \{k_n\}\subset J_{T_k}$.

We allow the domain $J$ to be the empty set but we then require
$n=0$ and we put $J^\bot=\emptyset$.

We denote the domain $J$ by $D(k)$. We call $k_i$ the
\textbf{vertices} (lying) on $k$, where $k_0$ is the \textbf{first}
vertex, $k_n$ is the \textbf{last} vertex, and $k_i, k_{i+1}$ are
\textbf{consecutive}. The \textbf{length} of $k$ equals $n$. We call
$I_k$ (resp.\ $T_k$) the \textbf{initial} (resp.\ the
\textbf{terminal}) marking of $k$. If $I_k=(J,k_0)$ (resp.\
$T_k=(J,k_n)$) we call it \textbf{trivial}.
\end{defin}

\begin{rem}[{compare \cite[Lemma 4.10]{MM}}]
\label{footprint} Let $k$ be a geodesic. Then for every spherical
subset $L\subset S$ there are at most two vertices from $L$ (lying)
on $k$, and if there are exactly two, then they are consecutive.
\end{rem}

\begin{defin}[{compare \cite[Definition 4.3]{MM}}]
Let $J\subset S$ be irreducible spherical. The set $J$ is a
\textbf{component domain} of a geodesic $b$ if for some $i$ we have
$D(b)\cup\{b_i\}=J$.

A component domain $J$ of a geodesic $b$ is \textbf{directly
subordinate backward} to $b$ (we denote this by $b\swarrow J$) if
$i>0$ or $I_b$ is not trivial.

A geodesic $k$ is \textbf{directly subordinate backward} to a
geodesic $b$ (we denote this by $b\swarrow k$) if
\begin{itemize}
\item
$b\swarrow D(k)$, \ and
\item
$I_k=
\begin{cases}
 (D(k),b_{i-1}) & \text{if } i>0,\\
 I_b & \text{if } i=0.
\end{cases}$
\end{itemize}

Analogously, a component domain $J=D(f)\cup\{f_i\}$ of a geodesic
$f$ of length $n$ is \textbf{directly subordinate forward} to $f$
(we denote this by $J\searrow f$) if $i<n$ or $T_f$ is not trivial.
A geodesic $k$ is \textbf{directly subordinate forward} to a
geodesic $f$ (we denote this by $k\searrow f$) if $D(k)\searrow f$,
and moreover $T_k= (D(k),f_{i+1})$ if $i<n$ or $T_k=T_f$ if $i=n$.
\end{defin}

\begin{defin}[{compare \cite[Definition 4.4]{MM}}]
\label{def hierarchy} A \textbf{hierarchy} is a set $H$ of geodesics
satisfying the following properties.
\begin{enumerate}[(i)]
\item
There is a distinguished \textbf{main} geodesic $g\in H$ with empty
domain (with a single vertex $s$ on $g$) and good initial and
terminal markings.
\item
For any irreducible spherical subset $J$ of $S$ with $b\swarrow
J\searrow f$, where $b,f\in H$, there is a unique geodesic $k\in H$
satisfying $D(k)=J$ and $b\swarrow k \searrow f$.
\item
For any geodesic $k\in H\setminus \{g\}$, there are geodesics $b,f\in H$ satisfying $b\swarrow k \searrow f$.
\end{enumerate}
\end{defin}

\begin{lemma}[{compare \cite[Theorem 4.6]{MM}}]
\label{hierachies exist} Assume that $S$ is twist-rigid. Then for any geodesic $g$ as in Definition
\ref{def hierarchy}(i), there is a hierarchy such that $g$ is its
main geodesic.
\end{lemma}
\begin{proof}
We follow the proof in \cite{MM}. We call a set $H$ of geodesics
satisfying properties (i), (iii), and the uniqueness part of
property (ii) a \textbf{partial hierarchy}. The set of partial
hierarchies in which $g$ is the main geodesic is non-empty, since
$\{g\}$ is a partial hierarchy. We claim that there exists a maximal
partial hierarchy in which $g$ is the main geodesic.

To justify the claim, it is enough to bound uniformly (above) the number of geodesics in any such
partial hierarchy $H$. We bound by induction on $i$ the number of
geodesics with domain of cardinality $i$. For $i=0$ there is only
one such geodesic, since for any such geodesic $k\in H$ we have by property (iii)
a sequence $g=b^n\swarrow \ldots \swarrow b^1\swarrow b^0=k$ of
geodesics with increasing domains, which implies $n=0$ and $k=g$.

The number of geodesics with domain of cardinality $i+1$ is bounded
by the square of the number of geodesics with domain of cardinality
$i$ times the number of irreducible spherical subsets of $S$ of
cardinality $i+1$: indeed, by property (iii) for each geodesic $k\in
H$ there are $b,f\in H$ satisfying $b\swarrow k\searrow f$ and by the
uniqueness part of property (ii) $b,f$ and $D(k)$ determine $k$
uniquely.

This proves the claim that there exists a maximal partial hierarchy in which $g$ is the main
geodesic.
\medskip

Now we prove that a maximal partial hierarchy $H$ is already a
hierarchy. Otherwise we would have some irreducible spherical subset
$J\subset S$ and geodesics $b,f\in H$ satisfying $b\swarrow
J\searrow f$, but no geodesic $k\in H$ with $D(k)=J$ and $b\swarrow
k\searrow f$.

Suppose $J=D(b)\cup b_i$. If $i>0$, then we denote
$K_I=\{b_{i-1}\}$. Otherwise, we put $K_I=J_{I_b}\setminus (J\cup
J^\bot)$ if it is non-empty and $K_I=\{m_{I_b}\}$ otherwise.
Similarly, suppose $J=D(f)\cup f_{i'}$, where the length of $f$
equals $n'$. If $i'<n'$, then we denote $K_T=\{f_{i'+1}\}$.
Otherwise, we put $K_T=J_{T_f}\setminus (J\cup J^\bot)$ if it is
non-empty and $K_T=\{m_{T_f}\}$ otherwise.

Since $S$ is twist-rigid, there is a geodesic path $(k_j)$ from some
element of $K_I$ to some element of $K_T$ in $S\setminus (J\cup
J^\bot)$ (possibly of length $0$). We define $I_k=(J,b_{i-1})$ in
case $i>0$ and put $I_k=I_b$ otherwise. Similarly, we define
$T_k=(J,f_{i'+1})$ in case $i'<n'$ and put $T_k=T_f$ otherwise.

Hence we have constructed a geodesic $k=((k_j),I_k,T_k)$ with domain
$J$ satisfying $b\swarrow k \searrow f$. Thus $H\cup\{k\}$ is a
partial hierarchy, which contradicts maximality of $H$. This proves
that a maximal partial hierarchy is a hierarchy and ends the proof
of the lemma.
\end{proof}

Note that although we have assumed that $S$ does not admit any
elementary twist, we have only used the fact that $S$ does not admit
an elementary twist with $J$ containing the fixed element $s$ of
$S$.

From now on, throughout this and the next section, we assume that we
are given a hierarchy $H$ with main geodesic $g$ as in the assertion
of Lemma~\ref{hierachies exist}.

\begin{defin} [{compare \cite[Section 4.3]{MM}}]
Let $J$ be a component domain. Then its \textbf{backward sequence}
is
$$\Sigma^-(J)=\{k\in H \colon D(k)\subset J \text{ and we have that } I_k \text{ is good or } m_{I_k}\notin J\}.$$
Its \textbf{forward sequence} is
$$\Sigma^+(J)=\{k\in H \colon D(k)\subset J \text{ and we have that } T_k \text{ is good or } m_{T_k}\notin J\}.$$
\end{defin}

\begin{lemma}[{compare \cite[Lemma 4.12]{MM}}]
\label{backward sequence} We have $\Sigma^-(J)=\{b^i\}_{i=0}^n$,
where the $b^i$ form a sequence $g=b^n\swarrow \ldots \swarrow b^0$.
If for some $b\in \Sigma^-(J)$ all the vertices on $b$ are outside
$J$, then $b=b^0$. An analogous statement holds for $\Sigma^+(J)$.
\end{lemma}
\begin{proof}
We follow again \cite{MM}. For the first assertion, since for every
geodesic $k\in H$ we have a sequence $g=b^n\swarrow \ldots \swarrow
b^0=k$, it is enough to prove the following.
\begin{enumerate}[(i)]
\item
If $k\in \Sigma^-(J)$ and $b\swarrow k$, then $b\in \Sigma^-(J)$.
\item
If $k,k'\in \Sigma^-(J)$ and $b\swarrow k,\ b\swarrow k'$, then
$k=k'$.
\end{enumerate}

(i) Let $k\in \Sigma^-(J)$ and $b\swarrow k$. Then $D(b)\subset
D(k)\subset J$. If $I_b$ is trivial, then so is $I_k$ and
$m_{I_k}\notin J$ is a vertex on $b$ which precedes the unique
element of $D(k)\setminus D(b)\subset J$. By Remark~\ref{footprint}
we have $m_{I_b}\notin J$. Hence $b\in \Sigma^-(J)$.

(ii) We prove this assertion together with the analogous one for
$\Sigma^+(J)$ by induction on $|J|$. Suppose $k,k'\in \Sigma^-(J)$
and $b\swarrow k,\ b\swarrow k'$. If $|J|=1$, \emph{i.e.}\ if
$J=\{s\}$, then $b=g$ and $k\searrow b, \ k'\searrow b$, hence by
the uniqueness part of property (ii) of a hierarchy we have $k=k'$.
Otherwise, let $j$ be the vertex from $J$ on $b$ nearest to $b_0$.
Since $k,k'\in \Sigma^-(J)$, we have $D(k)=D(b)\cup \{j\}=D(k')$. If
$|D(k)|<|J|$, then $k,k',b$ belong to $\Sigma^-(D(k))$ and we have
$k=k'$ by the induction hypothesis. Otherwise we have $J=D(k)$,
hence $k,k'\in \Sigma^+(J)$. By the induction hypothesis, by the
analogous statement for $\Sigma^+(J)$, we have $f\in H$ with
$k\searrow f, k'\searrow f$. By the uniqueness part of property (ii)
of a hierarchy we obtain $k=k'$.

\smallskip

For the second assertion note that if we have $b\swarrow k$ and
$k\in \Sigma^-(J)$, then the unique element of $D(k)\setminus D(b)$
which is a vertex on $b$ belongs to $J$.
\end{proof}

\begin{cor}[{compare \cite[Lemma 4.15]{MM}}]
\label{uniqueness} Let $J$ be a component domain with $H\ni
b\swarrow J$. Then $b$ is uniquely determined by $J$. Analogously,
if $J\searrow f\in H$, then $f$ is uniquely determined by $J$. In
particular, if $k\in H$, then $D(k)$ uniquely determines $k$.
\end{cor}
\begin{proof} If $b\swarrow J$, then $b\in \Sigma^-(J)$ and by Lemma
\ref{backward sequence} the cardinality of $D(b)$ determines $b$
uniquely. The last statement follows from the uniqueness part of
property (ii) of a hierarchy.
\end{proof}

\begin{prop}[{compare \cite[Lemma 4.21 and Theorem 4.7(3)]{MM}}]
\label{component domains are domains} Let $J$ be a component domain.
If $J\searrow f\in H$, then there is $b\in H$ with $b\swarrow J$
(and \emph{vice versa}). In particular, there is $k\in H$ with
$D(k)=J$.
\end{prop}
\begin{proof}
The second assertion follows from the existence part of property
(ii) of a hierarchy.

We prove Proposition~\ref{component domains are domains} by
induction on $|J|$. The case where $|J|=1$ is immediate. Otherwise,
let $k$ be the element of $\Sigma^-(J)$ with the largest domain.
First assume that there is a vertex from $J$ on $k$ and let $k_i\in
J$ be such a vertex with the least index $i$. We have $k\swarrow
D(k)\cup \{k_i\}$, hence we are done if $D(k)\cup \{k_i\}=J$
(actually, this cannot happen, because then the geodesic with
support $J$ is also in $\Sigma^-(J)$). Otherwise, we apply the
inductive hypothesis to $D(k)\cup \{k_i\}\subsetneq J$. We obtain a
geodesic $h\in H$ with domain $D(k)\cup \{k_i\}$ satisfying
$k\swarrow h$. Then we have $h\in \Sigma^-(J)$, which contradicts
the choice of $k$.

Now we consider the case where all the vertices on $k$ are outside
$J$. Then we have $k\in \Sigma^+(J)$ and, by Lemma~\ref{backward
sequence} applied to $\Sigma^+(J)$, the geodesic $k$ has the largest
domain among the geodesics in $\Sigma^+(J)$. Since we have $f\in
\Sigma^+(J)$, $|D(f)|=|J|-1$, and $k\neq f$, it follows that $D(k)$
equals $J$. By property (iii) of a hierarchy there is $b\in H$ with
$b\swarrow D(k)$.
\end{proof}

\section{Slices}
\label{slices}

In this section we show how to resolve a hierarchy into a sequence
of \emph{slices}, compare \cite[Section 5]{MM}. The slices give rise
to good markings related by moves N1, M2 and M3 and we conclude with
the proof of Theorem~\ref{connecting markings}.

We assume we are given a fixed hierarchy $H$ with main geodesic $g$
with a single vertex $s$. All markings and notation are as in the
previous section.

\begin{defin}
A \textbf{slice} is a pair $(k,m)$, where $k\in H$ is a geodesic and
$m$ is a vertex on $k$ such that $D(k)\cup \{m\}$ is non-spherical.

The marking \textbf{associated} to the slice $(k,m)$ is the pair
$(D(k),m)$. By Remark~\ref{extra remark} this marking is good.

We define the \textbf{initial slice} in the following way. Let
$k^0=g$. For $i\geq 0$, while the first vertex $k^i_0$ on the
geodesic $k^i$ does not equal $m_{I_{k^i}}=m_{I_g}$, we define
$k^{i+1}$ to be the geodesic in $H$ whose domain is $D(k^i)\cup
\{k^i_0\}$ --- its existence is guaranteed by Theorem~\ref{component
domains are domains} and its uniqueness by
Corollary~\ref{uniqueness}. The initial slice is the last geodesic
$k$ of this sequence together with its first vertex $m=m_{I_g}$.
Analogously we define the \textbf{terminal slice}.
\end{defin}

\begin{rem}
\label{associated markings} The good marking associated to the
initial slice equals $I_g$ (the initial marking of the main
geodesic). The marking associated to the terminal slice equals
$T_g$.
\end{rem}

\begin{defin}
\label{consecutive slices} We say that the slice $(k',m')$ is a
\textbf{successor} of the slice $(k,m)$ if we have one of the
following configurations:
\begin{enumerate}[(i)]
\item
$k=k'$ and $m,m'$ are consecutive vertices on $k$, or
\item
$k\swarrow k'$ and $m'=m$ is the first vertex on $k'$, or
\item
$k\searrow k'$ and $m=m'$ is the last vertex on $k$, or
\item
there is $h\in H$ satifying $k\searrow h \swarrow k'$,\ $m$ is the
last vertex on $k$, $m'$ is the first vertex on $k'$, and $m',m$ are
consecutive vertices on $h$.
\end{enumerate}
\end{defin}

\begin{rem}
\label{ends end} The terminal slice has no successor. The initial
slice is not a successor of any slice.
\end{rem}

\begin{thm}
\label{order good} For each slice which is not the terminal slice
there exists a unique successor. Each slice which is not the
initial slice is a successor of a unique other slice.
\end{thm}
\begin{proof}
Let $(k,m)$ be a slice. We prove that $(k,m)$ has a successor or is
the terminal slice.

We first assume that $m$ is not the last vertex on $k$. Let $m'$ be
the vertex on $k$ following the vertex $m$. If $D(k)\cup\{m'\}$ is
non-spherical, then $(k,m')$ is a slice. Slices $(k,m)$ and $(k,m')$
are in configuration (i) of Definition~\ref{consecutive slices}, in
particular $(k,m')$ is a successor of $(k,m)$. If $D(k)\cup\{m'\}$
is spherical, then we have $k\swarrow D(k)\cup\{m'\}$ and by
Theorem~\ref{component domains are domains} there is a geodesic
$k'\in H$ with $D(k')=D(k)\cup \{m'\}$ and $k\swarrow k'$. Then the
first vertex on $k'$ is $m$ and the slice $(k',m)$ is a successor of
$(k,m)$ in configuration (ii).

We now assume that $m$ is the last vertex on $k$. Since we have
$k\neq g$, there is $k'\in H$ satisfying $k\searrow k'$. We consider
first the case where $m$ is a vertex on $k'$. Let $m'$ be the unique
element of $D(k)\setminus D(k')$ which is the vertex on $k'$
preceding $m$. If $D(k')\cup \{m\}$ is non-spherical, then $(k',m)$
is a slice which is a successor of $(k,m)$ in configuration (iii).
Otherwise we have $k'\swarrow D(k')\cup \{m\}$ and by
Proposition~\ref{component domains are domains} there is a geodesic
$k''$ with $D(k'')=D(k')\cup \{m\}$ and $k'\swarrow k''$. The first
vertex on $k''$ equals $m'$. The pair $(k'',m')$ is a slice, since
$D(k'')\cup\{m'\}=D(k)\cup\{m\}$ is non-spherical. Then $(k'',m')$
is a successor of $(k,m)$ in configuration (iv).

It remains to consider the case where $m$ is not on $k'$. Consider
the geodesics satisfying $k=k^n\searrow k'=k^{n-1}\searrow\ldots
\searrow k^0=g$. Then for every $n\geq i> 0$ the unique element of
$D(k^i)\setminus D(k^{i-1})$ is the last vertex on $k^{i-1}$ and
$T_{k^{i-1}}$ is not trivial. By Corollary~\ref{uniqueness}, the
slice $(k,m)$ is the terminal slice. This completes the proof that
every slice has a successor or is the terminal slice.

Following the same scheme and using Corollary~\ref{uniqueness}
instead of Proposition~\ref{component domains are domains} we obtain
that a successor is unique. Analogously, every slice which is not
the initial slice is a successor of a unique other slice.
\end{proof}

In view of Remark~\ref{ends end}, Theorem~\ref{order good} has the
following immediate consequence.

\begin{cor}[{compare \cite[Proposition 5.4]{MM}}]
\label{resolution finite} There is a (unique) sequence of slices
from the initial slice to the terminal slice, such that for each
pair of consecutive elements, the second slice is a successor of the
first slice.
\end{cor}

We are now prepared for the following.

\begin{proof}[Proof of Theorem~\ref{connecting
markings}] Let $\mu,\mu'$ be two different good markings with core
$s\in S$. Since $S$ is twist-rigid, by Lemma~\ref{hierachies exist}
there is a hierarchy $H$ with main geodesic $g$ with a single vertex
$s$ and $I_g=\mu, T_g=\mu'$. By Remark~\ref{associated markings} and
Corollary~\ref{resolution finite} it is now enough to justify that
if a slice $(k',m')$ is a successor of a slice $(k,m)$, then their
associated good markings $\nu=(D(k),m),\ \nu'=(D(k'),m')$ are
related by move N1, M2 or M3.

If $(k,m), (k',m')$ are in configuration (i) of Definition
\ref{consecutive slices}, then $\nu,\nu'$ are related by move N1. If
$(k,m), (k',m')$ are in configuration (ii) or (iii), then $\nu,\nu'$
are related by move M2. Finally, if $(k,m), (k',m')$ are in
configuration (iv), then $\nu,\nu'$ are related by move M3.
\end{proof}

\begin{rem}
\label{removing assumptions hierarchies} Observe that in the above
proof we have only once used the hypothesis that $S$ is twist-rigid,
to guarantee the existence of the appropriate hierarchy
(Lemma~\ref{hierachies exist}). However, the proof of
Lemma~\ref{hierachies exist} just requires that $S$ does not admit
an elementary twist with $J$ containing $s$. Hence in the statement
of Theorem~\ref{connecting markings} we could replace the hypothesis
that $S$ is twist-rigid with the above weaker hypothesis. However,
we do not need this stronger result.
\end{rem}

\section{The last move}
\label{section 2-spherical complete markings}
In this section we complete the proof of the main theorem by proving Proposition~\ref{connecting 2}.
We consider bases and markings in full generality, as defined in Section~\ref{half-space choices}.

\begin{defin}
\label{shadow} Let $(s,w)$ be a base with support $J$. The
\textbf{shadow} of the base $(s,w)$ is the set of those elements
$j\in J$ which satisfy $d(wj.c_0,\mathcal Y_s)\leq d(w.c_0,\mathcal
Y_s)$. We denote the shadow by $\widetilde{J}$.
\end{defin}

\begin{lemma}
\label{Deodhar} The shadow $\widetilde{J}$ is spherical (possibly
reducible).
\end{lemma}
\begin{proof}
Let $I\subset \widetilde{J}$ be the set of $j\in \widetilde{J}$
satisfying $d(wj.c_0,\mathcal{Y}_s)=d(w.c_0,\mathcal Y_s)$. By
Lemma~\ref{brink-howlett}, the set $I$ commutes with the reflection
$r=wsw^{-1}$. By condition (i) in Definition~\ref{domain} we have
$\ell(r)=2\ell(w)+1$. Hence $r$ does not lie in any $W_K$ for a
proper subset $K\subset J$. Then Proposition~\ref{deod main}
guarantees that $I$ is spherical. Denote by $w_I$ the longest
element of $W_I$. Elements $j\in \widetilde{J}$ with
$d(wj.c_0,\mathcal{Y}_s)<d(w.c_0,\mathcal Y_s)$ satisfy
$\ell(rj)<\ell(r)$. Since $w_Ir=rw_I$, the shadow $\widetilde{J}$ is
contained in (and in fact equals) the set of $j\in J$ satisfying
$\ell(w_Irj)<\ell(w_Ir)$. Hence, by Lemma~\ref{bourbaki},
$\widetilde{J}$ is spherical.
\end{proof}

We have the following generalisation of Remark~\ref{extending
marking}(ii), which follows from Remark~\ref{extending marking}(i).

\begin{rem}
\label{extending avoiding shadow} If $((s,w),m)$ is a complete
marking, $j$ is an element of $S\setminus (\widetilde{J}\cup
J^{\bot})$ and $((s,w),j)$ is not a complete marking, then
$((s,wj),m)$ is a complete marking.
\end{rem}

Below we use the following terminology. Let $T$ be a subset of $S$.
A \textbf{component} of $T$ is maximal subset $T'\subset T$ such
that each two elements of $T'$ are connected by a path in $T$. A
subset $J\subset T$ \textbf{separates} $T$ if $T\setminus J$ has at
least two non-empty components. A subset $J\subset T$ \textbf{weakly
separates} $T$ if $J\cup J^\bot$ separates $T$. According to this
terminology, the set $S$ is twist-rigid if there is no irreducible
spherical subset $J\subset S$ which weakly separates $S$.

\begin{lemma}
\label{avoiding shadow} Assume that $S$ is twist-rigid. Let
$J\subset S$ be irreducible $2$-spherical and non-spherical. Let
$K\subset J$ be spherical (possibly reducible). Then for every $m\in
S \setminus (J\cup J^\bot)$ we have the following:
\begin{enumerate}[(i)]
\item
$m$ is in the same component of $S \setminus (K\cup J^\bot)$ as
$J\setminus K$, or
\item
$m$ is not adjacent to any element of $J\setminus K$, $m$ belongs to
$K^\bot$, and $J\cup \{m\}$ is twist-rigid.
\end{enumerate}
\end{lemma}

\begin{proof} We show that if any of the three elements of assertion (ii) does not hold, then we have assertion (i).
First, obviously if $m$ is adjacent to some element of $J\setminus K$, then we have assertion (i).

Second, if $m\notin K^\bot$, then we have $m\notin K'^\bot$ for some
irreducible $K'\subset K$ satisfying $K\subset K'\cup K'^\bot$.
Since $J$ is irreducible, $J\setminus (K'\cup K'^\bot)$ is
non-empty. Since $K'$ does not weakly separate $S$, there is a path
from $m$ to an element of $J\setminus (K'\cup K'^\bot)$ outside
$K'\cup K'^\bot\supset K\cup J^\bot$, and we have assertion (i).

Otherwise, if $m\in K^\bot$ but $J\cup \{m\}$ is not twist-rigid,
then there exists some irreducible spherical subset $L\subset
J\cup\{m\}$ which weakly separates $J\cup\{m\}$. We must have
$m\notin L\cup L^\bot$ and $K\subset L\cup L^\bot$, because $J$ is
$2$-spherical. Since $L$ does not weakly separate $S$, there is a
path from $m$ to some element of the non-empty set $J\setminus
(L\cup L^\bot)\subset J\setminus K$ outside $L\cup L^\bot\supset
K\cup J^\bot$. This again yields assertion (i).
\end{proof}

\begin{lemma}
\label{small twist rigid} In the case of assertion (ii) in Lemma
\ref{avoiding shadow}, the set $K$ is a maximal spherical subset of
$J$.
\end{lemma}
\begin{proof}
If there is a spherical subset $L\subset S$ with $K\subsetneq
L\subset J$, then we have $L\neq J$, since $J$ is non-spherical. Let
$L'\subset L$ be irreducible satisfying $L\subset L'\cup L'^\bot$
and containing an element outside $K$. Then $L'\cup L'^\bot$ does
not contain $m$, contains $L\supset K$, but does not contain some
other vertex in $J$, by irreducibility of $J$. Hence $L$ weakly
separates $J\cup \{m\}$. Contradiction.
\end{proof}

We are now ready for the following.

\begin{proof}[Proof of Proposition~\ref{connecting 2}]
Let $\mu=((s,w),m)$ be the complete marking with non-spherical
support $J$ which we want to relate by moves to some semi-complete
marking $\mu'$ with support $J'$ containing $J$. We prove
Proposition~\ref{connecting 2} by (backward) induction on $\ell(w)$.
By Remark~\ref{parallel wall theorem}, for $\ell(w)$ large enough
the content of Proposition~\ref{connecting 2} is empty. Suppose we
have verified Proposition~\ref{connecting 2} for $\ell(w)=k+1$.
Assume now $\ell(w)=k$. By Lemma~\ref{Deodhar}, the shadow
$\widetilde{J}\subset J$ of $(s,w)$ is spherical.

Since $S$ is twist-rigid, we are in position to apply Lemma~\ref{avoiding shadow}, with
$K=\widetilde{J}$. First assume that we are in the case of assertion
(ii) of Lemma~\ref{avoiding shadow} and thus we also have the
conclusion of Lemma~\ref{small twist rigid}. Let $\mu'=((s,w'),
m')$ be any marking with support $J'$ satisfying $J'\cup \{m'\}=J$. The marking $\mu'$ is semi-complete
since $J$ is irreducible non-spherical. Then $\mu$ and $\mu'$ are
related by move M4 and we are done.

Now assume that we are in the case of assertion (i) of Lemma
\ref{avoiding shadow}. Then there is a path $(h_0=m, h_1, \ldots,
h_l)$ in $S\setminus (\widetilde{J}\cup J^\bot)$, where $h_l\in
J\setminus \widetilde{J}$. Let $i$ be the least index such that
$J\cup\{h_i\}$ is $2$-spherical (possibly $i=l$). Then for $1\leq
i'<i$ the complete markings $((s,w), h_{i'-1}), ((s,w),
h_{i'})$ are related by move M1.

If $((s,w), h_i)$ is complete, then also $((s,w), h_{i-1}), ((s,w),
h_i)$ are related by move M1. Then we can put $\mu'=((s,w), h_i)$.
If $((s,w), h_i)$ is not complete, then $((s,w), h_{i-1})$ is
related by move M2 to $\nu=((s,wh_i), h_{i-1})$, which is a complete
marking by Remark~\ref{extending avoiding shadow}.

The word-length of $wh_i$ equals $k+1$ and we can apply
Proposition~\ref{connecting 2} with $\ell(w)=k+1$. We obtain that
$\nu$ is related by a sequence of moves M1, M2 and M4 (only allowed
as the last move) to a marking $\mu'=((s,w'), m')$ with support $J'$
such that $J'\cup\{m'\}$ is $2$-spherical and satisfies $J'\supset
J\cup\{h_i\}\supset J$. This finishes the proof of
Proposition~\ref{connecting 2} for $n=k$.
\end{proof}

\begin{rem}
\label{removing assumptions hierarchies 2} In the above argument, we
have only once used the hypothesis that $S$ is twist-rigid, in the
proof of Lemma~\ref{avoiding shadow}. However, we could just require
that there is no irreducible spherical subset $K\subset S$ which
weakly separates $S$ and together with $s$ is contained in some
irreducible $2$-spherical non-spherical subset $J\subset S$. This
allows to weaken the hypothesis of Proposition~\ref{connecting 2}.
Again, we do not need this stronger result.
\end{rem}

\appendix

\section{Reflection- and angle-compatibility}
\label{App1}

In this appendix we prove that the relations of reflection- and angle-compatibility are symmetric.
Let $R$ be a Coxeter generating set for a group $W$ and let $\mathbb A$ be the associated Davis complex.

Translating the language of the root systems into the language of
the Davis complex, the main result of \cite{Deodhar_reflections} may
be phrased as follows.

\begin{thm}
\label{small Coxeter groups} Let $S\subset W$ be some set of
$R$-reflections and let $W_S \subset W$ be the subgroup generated by
$S$. Define $\bar S$ as the set of all conjugates under $W_S$ of
elements of $S$. Let $C \subset \mathbb A$ denote a connected
component of the space obtained by removing from $\mathbb A$ every
wall associated to an element of $\bar S$. Let $S_C$ be the subset
of $\bar{S}$ of $R$-reflections in walls adjacent to some chamber in
$C$.

Then $S_C$ is a Coxeter generating set for $W_S$ and (the closure
of) $C$ is a fundamental domain for the $W_S$-action on $\mathbb A$.
In particular $W_S$ is a Coxeter group.
\end{thm}

\begin{cor}\label{cor:ReflCompatible}
Let $S\subset W$ be a Coxeter generating set such that every element of $S$ is conjugate to some element of $R$. Then every element of $R$ is conjugate to some element of $S$.
\end{cor}
\begin{proof}
Indeed, if $W_S = W$, then the set $C$ consists of exactly one chamber.
\end{proof}

In order to obtain a similar statement concerning
angle-compatibility, we record the following well-known fact.

\begin{lemma}\label{lem:2refl}
Given a pair of $R$-reflections $\{s, s'\}\subset W$ generating a
finite subgroup, there is a spherical pair $\{r, r'\} \subset R$
such that $W_{\{s, s'\}}$ is conjugate to a subgroup of $W_{\{r,
r'\}}$.
\end{lemma}

In other words, every finite reflection subgroup of rank $2$ is contained in a finite parabolic subgroup of rank $2$.

\begin{proof}
Since $V = W_{\{s, s'\}}$ is finite, it is contained in some finite
parabolic subgroup of $W$. We may thus assume without loss of
generality that $W$ is finite. Let $\mathbb S$ denote the underlying
sphere of the corresponding Coxeter complex. The fixed point set
$\mathbb S^V$ of $V$ in $\mathbb S$ has codimension~at most $2$,
since it contains the intersection of two equators. Therefore, the
parabolic subgroup generated by all the reflections fixing $\mathbb
S^V$ pointwise contains $V$ and has rank~at most $2$.
\end{proof}

\begin{cor}\label{cor:AngleCompatible}
Let $S\subset W$ be a Coxeter generating set such that every
spherical pair of elements of $S$ is conjugate to some pair of
elements of $R$. Then every spherical pair of elements of $R$ is
conjugate to some pair of elements of $S$.
\end{cor}
\begin{proof}
By Corollary~\ref{cor:ReflCompatible}, every element of $R$ is an
$S$-reflection. Therefore, given a pair $\{r, r'\} \subset R$,
Lemma~\ref{lem:2refl} (with the roles of $S$ and $R$ interchanged)
yields a pair $\{s, s'\} \subset S$ such that $W_{\{r, r'\}}$ is
conjugate to a subgroup of $W_{\{s, s'\}}$. By hypothesis the pair
$\{s, s'\}$ is conjugate to some pair $\{t, t'\} \subset R$. In
particular $W_{\{r, r'\}} \subset w W_{\{t, t'\}} w\inv$ for some $w
\in W$. Since $W_{\{r, r'\}}$ and $W_{\{t, t'\}}$ are parabolic
subgroups of the same rank, we deduce successively that we have
$W_{\{r, r'\}} = w W_{\{t, t'\}} w\inv$ and then $ \{r, r'\} = u\{t,
t'\} u\inv$ for some $u \in W$. The result follows since the pairs
$\{s, s'\}$ and $\{t, t'\}$ are conjugate.
\end{proof}

\section{Reflection- and angle-deformations of twist-rigid Coxeter generating sets}
\label{App2}

The goal of this appendix is to prove Corollary~\ref{corollaries to
main theorem}. We present the basic facts from
\cite{HowlettMuhlherr} and \cite{MarquisMuhlherr} needed for that
purpose.

Let $S$ be a Coxeter generating set for $W$. Following
\cite[Definition 5]{HowlettMuhlherr}, we say that an element $\tau
\in S$ is a \textbf{pseudo-transposition} if there is some $J\subset
S$ such that the following conditions hold.
\begin{description}
\item[PT1] The set $J$ contains $\tau$ and for every $s \in S \setminus J$ either $s$ and $\tau$ are not adjacent or $s$ belongs to $J^
\bot$.

\item[PT2] There is an odd number $k$ such that $J$ is of type $C_k$ or $I_2(2k)$, and in the first case
$\tau$ is the unique element of $J$ commuting with all other
elements of $J$ except for one with which $\tau$ generates the
dihedral group of order $8$.
\end{description}

Suppose that $\tau$ is a pseudo-transposition.
We then define $\rho$ to be the longest word $w_J$ of $W_J$, which
is an involution and is central in $W_J$. Let $a$ be the unique
element of $J$ different from $\tau$ and not commuting with $\tau$.
We set $\tau' = \tau a \tau$. Finally, we define $S' = S \cup
\{\tau', \rho\}\setminus \{\tau\} $.
It is shown in \cite[Lemma~6]{HowlettMuhlherr} that $S'$ is also a
Coxeter generating set. We say that $S'$ is an \textbf{elementary
reduction} of $S$.

\begin{lemma}\label{lem:reduction:twist-rigid}
Let $S'$ be an elementary reduction of $S$. Then $S$ is twist-rigid
if and only if $S'$ is twist-rigid.
\end{lemma}

\begin{proof}


For $L\subset S\cup S'$, we denote, exceptionally, by $L^\bot$ the
set of all elements of $S\cup S'\setminus L$ commuting with $L$.
Note that if $L\subset S$ (resp.\ $L\subset S'$), then the set
$S\setminus (L\cup L^\bot)$ (resp.\ $S'\setminus (L\cup L^\bot$)) is
independent of whether we consider the usual or the exceptional
definition of $L^\bot$.

\medskip

Assume first that $S$ is twist-rigid and let $L \subset S'$ be an
irreducible spherical subset. We have to show that $S'\setminus
(L\cup L^\bot)$ is \textbf{connected}, \emph{i.e.}\ that it has only
one connected component (for this and the other definitions see
Section~\ref{section 2-spherical complete markings}).


We denote $J'=J\cup \{\tau'\}\setminus \{\tau\} $, which is an
irreducible spherical subset of $S'$. The set $K=J'\setminus
\{\tau'\}=J\setminus \{\tau\}$ is also irreducible. Since $S$ is
twist-rigid, $K$ does not weakly separate $S$ and hence $\tau$
belongs to the unique connected component of $S\setminus (K\cup
K^\bot)$. Since all elements of $S$ adjacent to $\tau$ lie in $K\cup
K^\bot$, we have $S\setminus (K\cup K^\bot)=\{\tau\}$ and
consequently $S'\setminus (K\cup K^\bot)=\{\tau'\}$.

Thus every element of $K$ is adjacent to every other element of
$S'$. Hence there is no loss of generality in assuming $K \subset L
\cup L^\bot$. Since $K$ is irreducible, there are two cases to
consider: either we have $K \subset L$ or $K\subset L^\bot$.

If $K \subset L$, then either we have $\tau\in L$ which implies
$L=J'$, or else, in view of $S' = J' \cup K^\bot$, we have $L = K$.
In the latter case $S' \setminus (L \cup L^\bot)$ is a singleton. In
the former case we have $S' \setminus (L \cup L^\bot) = S \setminus
(J \cup J^\bot)$ and this set is connected since $S$ is twist-rigid.
Thus, if $K \subset L$, then $S' \setminus (L \cup L^\bot)$ is
connected, as desired.

If $K \subset L^\bot$, we first assume $\rho\in L$. Then
$L=\{\rho\}$ and we are done because $S' \setminus (L \cup L^\bot) =
S' \setminus (J' \cup J'^\bot) = S \setminus (J \cup J^\bot)$, which
is connected since $S$ is twist-rigid. We now assume $\rho\in
L^\bot$. Then we also have $\tau'\in L^\bot$ which implies $L\subset
S$. Moreover, then the set $S' \setminus (L \cup L^\bot)$ coincides
with $S\setminus (L \cup L^\bot)$, which is connected since $S$ is
twist-rigid, and we are done.

Finally, it remains to consider the case where $K \subset L^\bot$
and $\rho$ belongs to $S' \setminus (L \cup L^\bot)$. Then we also
have $\tau'\in S'\setminus (L \cup L^\bot)$, hence $L$ is contained
$S$. It suffices to show that $S'\setminus (L\cup L^\bot\cup
\{\rho\})$ is connected. The bijection from $S'\setminus (L\cup
L^\bot\cup \{\rho\})$ onto $S\setminus (L\cup L^\bot)$, which maps
$\tau'$ to $\tau$ and restricts to the identity outside $\{\tau'\}$,
preserves the adjacency relation. Hence the connectedness of
$S'\setminus (L\cup L^\bot\cup \{\rho\})$ follows from the
connectedness of $S\setminus (L \cup L^\bot)$.

\medskip
Similar arguments show that, conversely, if $S'$ is twist-rigid and
$L$ is an irreducible spherical subset of $S$, then $S \setminus (L
\cup L^\bot)$ is connected.
\end{proof}

A Coxeter generating set is called \textbf{reduced} if it does not
contain any pseudo-transposition. For any Coxeter generating set $S$
there is a sequence $S = S_1, \dots, S_n$ of Coxeter generating
sets, where $n\leq |S|$, such that every $S_{i+1}$ is an elementary
reduction of $S_i$, and $S_n$ is reduced (\cite[Proposition
7]{HowlettMuhlherr}).



\begin{thm}[{\cite[Theorem 1]{HowlettMuhlherr}}]
\label{alg-reflection} Let $R$ be a reduced Coxeter generating set
for $W$. There is an explicit finite subgroup $\Sigma \subset
\Aut(W)$ such that for each reduced Coxeter generating set $S$ for
$W$, there is some $\alpha \in \Sigma$ such that $\alpha(S)$ and $R$
are reflection-compatible.
\end{thm}

This result is supplemented by the following. We use freely the
terminology of \cite{MarquisMuhlherr}.
\begin{thm}[{\cite[Theorem 2]{MarquisMuhlherr}}]
\label{alg-angle} Let $S$ and $R$ be reflection-compatible Coxeter
generating sets for $W$. Then there is a sequence $S = S_1, \dots,
S_n$ of Coxeter generating sets, where $n\leq |S|$, such that every
$S_{i+1}$ is a $J_i$-deformation of $S_i$, for some $J_i\subset
S_i$, and $S_n$ and $R$ are angle-compatible.
\end{thm}

Although in general $J_i$-deformations do not have to extend to
automorphisms of $W$, this is in fact the case if $S$ is
twist-rigid.

\begin{lemma}\label{lem:AngleDeformation}
Let $S$ be a twist-rigid Coxeter generating set for $W$. Then any
$J$-deformation of $S$ extends to an automorphism of $W$.
\end{lemma}

\begin{proof}
The only $J$-deformation which might \emph{a priori} not extend to
an automorphism of $W$ is described in
\cite[Section~7.6]{MarquisMuhlherr}. By
\cite[Definition~7.6]{MarquisMuhlherr}, the sets $\{s\}$ and
$J^\infty \setminus \{r\}^\perp$ fall into two distinct connected
components of $S \setminus (\{r\} \cup \{r\}^\perp)$. Thus if $S$ is
twist-rigid, then $ J^\infty \subset \{r\}^\perp$.

Moreover, we have $S  = K \cup J^\perp \cup J^\infty$ by
\cite[Lemma~7.7]{MarquisMuhlherr} and every vertex of $J^\perp$
adjacent to some vertex of $J^\infty$ actually belongs to
$\{t\}^\perp$ by Condition~(TWt) from
\cite[Definition~7.6]{MarquisMuhlherr}. It follows that the sets
$\{s\}$ and $J^\infty \setminus \{t\}^\perp$ fall into two distinct
connected components of $S \setminus (\{t\} \cup \{t\}^\perp)$. Thus
if $S$ is twist-rigid, then $ J^\infty \subset \{t\}^\perp$.

We infer that if $S$ is twist-rigid,
then $J^\infty$ must be contained in $\{r,t\}^\perp$. In view of
\cite[Lemma~7.14]{MarquisMuhlherr} the Coxeter generating set
$S$ and its $J$-deformation $\delta(S)$ have the same Coxeter matrix.
\end{proof}

We are now ready for the following.

\begin{proof}[Proof of Corollary~\ref{corollaries to main theorem}]
Let $R$ be a twist-rigid Coxeter generating set for $W$. A sequence
of elementary reductions transforms $R$ into a reduced Coxeter
generating set $R'$. By Lemma~\ref{lem:reduction:twist-rigid} the
set $R'$ is twist-rigid.

Let now $S$ be any other Coxeter generating set for $W$. Let $S'$ be
a reduced Coxeter generating set obtained from $S$ by a sequence of
elementary reductions. By Theorem~\ref{alg-reflection} there is an
automorphism $\alpha \in \Sigma\subset\Aut(W)$ such that
$\alpha(S')$ and $R'$ are reflection-compatible. By
Theorem~\ref{alg-angle} and Lemma~\ref{lem:AngleDeformation} there
is a sequence of $J$-deformations which extend to automorphisms of
$W$ transforming $\alpha(S')$ into a Coxeter generating set $S''$
such that $S''$ and $R'$ are angle-compatible. By Theorem
\ref{main}, the set $R'$ is conjugate to $S''$ and consequently also
to $\alpha(S')$ and to $S'$. Then by Lemma
\ref{lem:reduction:twist-rigid} the set $S$ is twist-rigid. This
proves assertion (i).

A conjugate of the set $S$ can be obtained from $R'$ by composing
with an element of $\Sigma$ and an \emph{a priori} bounded number of
$J$-deformations and operations inverse to elementary reductions.
This yields assertion (ii) and in particular assertion (iii).
\end{proof}

\end{document}